\documentclass{amsart}

\usepackage{amsmath}
\usepackage{amssymb}
\usepackage{amsthm}
\usepackage{mathabx}
\usepackage{amsrefs}        
\usepackage{pgfplots}
\usepackage{tikz}
\usetikzlibrary{trees}
\usetikzlibrary{arrows}
\usepackage{hyperref}
\usepackage[T2A]{fontenc}
\usepackage[utf8]{inputenc}
\usepackage{bbm}
\usepackage{bbold}
\usepackage{datetime}
\usepackage{comment}
\DeclareMathAlphabet{\mathbbold}{U}{BOONDOX-ds}{m}{n}
\DeclareMathOperator\supp{supp}
\DeclareMathOperator{\folner}{\text{F{\o}l}}

\newfont{\bbf}{msbm10 scaled\magstep1}

\def\arxivprefix{https://arxiv.org}

\newcounter{glob}[section]
\renewcommand\theglob{%
	\ifnum\arabic{section}=0\else\arabic{section}.\fi %
	\arabic{glob}}

\theoremstyle{plain}
\newtheorem{thm}[glob]{Theorem}
\newtheorem{lemma}[glob]{Lemma}
\newtheorem{cor}[glob]{Corollary}
\newtheorem{prop}[glob]{Proposition}
\newtheorem*{prop*}{Proposition}

\theoremstyle{definition}
\newtheorem{defi}[glob]{Definition}

\theoremstyle{remark}
\newtheorem{remark}[glob]{Remark}

\def\Z{\mathbb{Z}}

\def\R{\mathbb{R}}

\def\N{\mathbb{N}}

\pgfplotsset{width=8cm}

\newdateformat{UKvardate}{%
	\THEDAY\ \monthname[\THEMONTH], \THEYEAR}
\UKvardate

\title{\bf Coulhon  Saloff-Coste isoperimetric inequalities for finitely generated groups}
\author{Christophe Pittet}
\address{I2M, CNRS, AMU, and University of Geneva}
\email{pittet@math.cnrs.fr}
\author{Bogdan Stankov}
\address{D\'epartement de math\'ematiques et applications, ENS, CNRS, PSL University, Paris}
\email{bogdan.zl.stankov@gmail.com}
\thanks{The authors acknowledge support of the FNS grant 200020-200400.}
\date{\today}


\makeatletter
\hypersetup{
pdfauthor={Christophe Pittet and Bogdan Stankov},
pdftitle=\@title,
pdfsubject=\@title,
pdflang={en-US}
}
\makeatother


\begin{document}
\maketitle 
\begin{abstract}  We prove an inequality, valid on any finitely generated group with a fixed finite symmetric generating set, involving the growth of successive balls, and the average length of an element in a ball. It generalizes recent improvements of the Coulhon Saloff-Coste inequality. We reformulate the inequality in terms of the F{\o}lner function; in the case the finitely generated group is amenable with exponential growth, this allows us to express the best possible (outer) constant in the Coulhon Saloff-Coste isoperimetric inequality with the help of a formula involving the growth rate and the asymptotic behavior of the F{\o}lner function. 
\end{abstract}

\thispagestyle{empty}

\section{Introduction}
In many spaces satisfying a Poincaré inequality (like for example Lie groups with a left-invariant Riemannian metric, or more generally Riemannian manifolds with a cocompact group of isometries, or discrete spaces like Cayley graphs of finitely generated groups, etc.), Coulhon Saloff-Coste inequalities \cite{CouSal} bring a lower bound on the  measure of the boundary of \emph{any} measurable subset (with regular boundary) of (large) volume $V$, in term  of the radii of the balls of  volume $V$.

In Section \ref{birth}, we shortly recall why this kind of inequalities appeared to be crucial in the work of Varopoulos related to Kesten recurrence conjecture on finitely generated groups, and mention some of their links with geometric group theory. With the exception of the precise statement of the original inequality of Coulhon and Saloff-Coste in Section \ref{birth}, the next sections are independent of Section \ref{birth}. 

In Section \ref{main}, we state Theorem \ref{thmaveragegrowth} which is the main result of the paper. It is valid on any finitely generated group, it involves the growth of successive balls, and the average length of an element in a ball. Its proof is based on a ``mass transportation'' idea proposed by Gromov \cite[pages 347-348]{Gro}.  It generalizes recent improvements of the original Coulhon Saloff-Coste inequality for finitely generated groups, due to Correia \cite{Cor}, Correia and Troyanov \cite{CorTro}, and Pete \cite{Pet}. In Proposition \ref{equal}, we reformulate the inequality in terms of the F{\o}lner function. In Theorem \ref{quotient}, we explain how this reformulation allows us to apply a formula, expressing, for any given finitely generated amenable group $G$ of exponential growth with a fixed finite symmetric generating set $S$, the best possible (outer) constant $C(G,S)$, in the Coulhon Saloff-Coste isoperimetric inequality. In Theorem \ref{thmc0}, we mention a estimate  of $\inf C(G,S)$, where the infimum is taken over all couples $(G,S)$ where $G$ is a finitely generated group and $S$ is  a finite symmetric generating set of $G$. 

Section \ref{proofs} is devoted to the proofs.

\section{On the birth of the Coulhon Saloff-Coste inequality}\label{birth}

P\'{o}lya has proven \cite{Pol} that the simple random walk on the $d$-dimensional grid $\mathbb Z^d$ is transient if and only if $d\geq 3$. In other words, the return probabilities $$p^{\mathbb Z^d}_{2n}(x_0,x_0)$$ to a starting point $x_0$ after
$2n$ closest neighborhood random steps in $\mathbb Z^d$ satisfy
\[
	\sum_{n=0}^{\infty}p^{\mathbb Z^d}_{2n}(x_0,x_0)<\infty,
\]
if and only if $d\geq 3$. Using F{\o}lner isoperimetric characterization of amenability \cite{Fol}, Kesten has proven \cite{KesFull} that the simple random walk on a group $G$ generated by a finite symmetric set satisfies
\[
\exists C>1, \exists a>0 : \forall n\in\mathbb N,\, p^G_{2n}(x_0,x_0)<Ce^{-an},	
\]
if and only if $G$ is non-amenable. He also raised  questions about recurrent groups \cite{KesCong}, leading to the following conjecture: the simple random walk on a group $G$ generated by a finite symmetric set is recurrent if and only if $G$ contains a finite index subgroup isomorphic to $\mathbb Z^d$ with $0\leq d\leq 2$. Varopoulos (see \cite{Varo} and \cite{CSV}) proved that if the growth of a finitely generated group is bounded below by a polynomial of degree $d$, then the simple random walk on it satisfies
\[
\exists C>1 : \forall n\in\mathbb N,\, p^G_{2n}(x_0,x_0)<Cn^{-d/2}.	
\]
This, combined with  Gromov  polynomial growth theorem, solves Kesten  conjecture. Indeed, suppose $G$ is a finitely generated group, not virtually $\mathbb Z^d$ with $d\leq 2$. According to the strong version of Gromov  result \cite{Klei}, the growth of $G$ must be at least polynomial of degree $3$. Applying Varopoulos implication cited above, we see that
\[
\exists C>1 : \forall n\in\mathbb N,\, p^G_{2n}(x_0,x_0)<Cn^{-3/2},
\]
so that 
\[
	\sum_{n=0}^{\infty}p^{G}_{2n}(x_0,x_0)<\infty.
\]
This proves Kesten  conjecture. Isoperimetric inequalities are useful to understand, to prove, and to generalize Varopoulos implication. This line of thought opened by Varopoulos, has been followed by Coulhon and Saloff-Coste \cite{CouSal} who have proved
that on an infinite group $G$ generated by a finite symmetric set $S$, any non-empty finite subset $\Omega$ of $G$, satisfies
$$\frac{|\partial_S\Omega|}{|\Omega|}\geq\frac{1}{4|S|\Phi_S[2|\Omega|]},$$
where $\partial_S\Omega$ is the boundary of $\Omega$ relative to $S$, the function $\Phi_S$ is the generalized inverse of the growth of $G$ relative to $S$, and vertical lines around a set denote its cardinality. See Section \ref{main} below for the precise definitions. A short proof of Varopoulos implication based on the Coulhon Saloff-Coste inequality is obtain by applying inequalities (1.3) and (1.4) and the equivalence (1.5) from  \cite{BPS}. See also \cite{PitSal} for  relations between growth, isoperimetric profiles, and return probabilities. See  \cite{PitSalWreath}, \cite{ErsZhe}, \cite{BriZheng}, for examples of groups with these asymptotic invariants evaluated.

\section{Main results and questions}\label{main}
\subsection{Improving the Coulhon Saloff-Coste isoperimetric inequality}
Consider a couple $(G,S)$, where $G$ is a finitely generated group, and $S\subset G$  is a finite symmetric generating subset of $G$.
For any $x\in G$, its \emph{word norm associated to $S$} is
\[
|x|_S=\min\{n:\exists g_1,\dots,g_n\in S:x=\prod_{k=1}^ng_k\}.	
\]
The \emph{left-invariant word metric associated to $S$} is defined, for any $x,y\in G$, as
$$d_S(x,y)=|x^{-1}y|_S.$$
We denote by  $e\in G$ the identity element. 
The \emph{ball of radius $r$, centered at $e$, with respect to the metric $d_S$}, is 
$$B_S(r)=\{x\in G:|x|_S\leq r\}.$$
Notice that $B_S(0)=\{e\}$ (the empty product equals $e$, hence $|e|_S=0$), and $B_S(1)=S$.
We agree that $B_S(-r)$ is empty for any $r>0$.
If $\Omega$ is a subset of $G$, we denote by $|\Omega|$ its cardinality.
For any $r\geq 0$, we define the \emph{average length of the elements of $B(r)$} as
\[
	\mathbb E\left[|X_r|_S\right]=\frac{1}{|B_S(r)|}\sum_{g\in B_S(r)}|g|_S.
\]
(It can be thought of as the expectation of the random variable
\[
	\omega\mapsto|X_r(\omega)|_S,
\]
where  the law of $X_r$ is the uniform distribution on $B_S(r)$.)
We consider the \emph{generalized inverse}
$$\Phi_S[v]=\inf\{r:|B_S(r)|>v\},$$
\emph{of the growth function of $(G,d_S)$}.
Notice that locally, the function $r\mapsto|B_S(r)|$, never increases (it is upper semicontinuous and takes integral values only), and the infimum in the above definition (which is finite for all $v\in\R$ if and only if $G$ is infinite) is a minimum.
As $d_S$ takes only integral values, all the finite values of $\Phi_S$ are also integral.
The (inner) \emph{boundary of $\Omega$ with respect to $S$} is
$$\partial_S\Omega=\{x\in\Omega:\exists s\in S:xs\in G\setminus\Omega\}=\{x\in\Omega: d_S(x,G\setminus\Omega)=1\}.$$

\begin{thm}[Isoperimetric inequalities in terms of the average length and the growth of successive balls]\label{thmaveragegrowth}
Let $(G,S)$ be a finitely generated group with a  finite symmetric generating set.
Let $\alpha\geq 0$.
For any finite non-empty subset $\Omega$ of $G$, such that $\Phi_S[(1+\alpha)|\Omega|]<\infty$ (which is always the case if $G$ is infinite), we have:
$$\frac{|\partial_S\Omega|}{|\Omega|}\geq\frac{\alpha}{1+\alpha}\frac{|B_S(r-1)|}{|B_S(r)|}\frac{1}
{\mathbb E\left[|X_r|_S\right]},$$
where $r=\Phi_S[(1+\alpha)|\Omega|]$.
\end{thm}

\begin{cor}[Isoperimetric inequalities in terms of the inverse of the growth of balls and the growth of successive balls]\label{thmgrowth}
Let $(G,S)$ be a finitely generated group with a  finite symmetric generating set.
Let $\alpha\geq 0$.
For any finite non-empty subset $\Omega$ of $G$, such that $\Phi_S[(1+\alpha)|\Omega|]<\infty$ (which is always the case if $G$ is infinite), we have:
$$\frac{|\partial_S\Omega|}{|\Omega|}\geq\frac{\alpha}{1+\alpha}\frac{|B_S(\Phi_S[(1+\alpha)|\Omega|]-1)|}{|B_S(\Phi_S[(1+\alpha)|\Omega|])|}\frac{1}{\Phi_S[(1+\alpha)|\Omega|]}.$$
\end{cor}

\begin{proof} The corollary follows from Theorem \ref{thmaveragegrowth} because for any $r\geq 0$,
	\[
		\mathbb E\left[|X_r|_S\right]=\frac{1}{|B_S(r)|}\sum_{g\in B_S(r)}|g|_S\leq r.
	\]
	
\end{proof}

As we will explain, the next statement, proved in \cite{CorTro}, follows from the above theorem and its proof: 

\begin{thm}[Isoperimetric inequalities with an outer multiplicative constant almost equal to $1$]\label{thmepsilon}
Let $(G,S)$ be a finitely generated group with a finite symmetric generating set.
Assume $G$ is infinite.
Let $0<\varepsilon<1$.
For any finite non-empty subset $\Omega$ of $G$ we have:
$$\frac{|\partial_S\Omega|}{|\Omega|}>(1-\varepsilon)\frac{1}{\Phi_S\left[\frac{1}{\varepsilon}|\Omega|\right]}.$$
\end{thm}

Choosing $\varepsilon=\frac{1}{2}$, we recover the following inequality, proved by Pete \cite[Theorem 5.11]{Pet} and Correia \cite{Cor}:

\begin{cor}[Improving the Coulhon Saloff-Coste inequality]\label{maincor}
Let $(G,S)$ be a finitely generated group with a finite symmetric generating set.
If $G$ is infinite, then for any non-empty finite subset $\Omega$ of $G$,
$$\frac{|\partial_S\Omega|}{|\Omega|}>1/2\frac{1}{\Phi_S[2|\Omega|]}.$$
\end{cor} 

The above inequality is an improvement of the original isoperimetric inequality of Coulhon and Saloff-Coste \cite[Théorème~1]{CouSal} stated above in the introduction.

\subsection{Searching optimal constants with the help of F{\o}lner functions}

\begin{defi}[The optimal multiplicative outer constant in the Coulhon Saloff-Coste inequality]
Let $(G,S)$ be a finitely generated group with a finite symmetric generating set.
We define 
$$C(G,S)=\sup\left\{c\geq 0:\exists \alpha\geq 0\, \mbox{such that}\,\forall\Omega\subset G,\frac{|\partial_S\Omega|}{|\Omega|}\geq c\frac{1}{\Phi_S[(1+\alpha)|\Omega|]}\right\},$$
where  $\Omega$ is always assumed to be finite and non-empty.
\end{defi}
We claim that $C(G,S)\geq 1$. Indeed: if $G$ is infinite, then  Theorem~\ref{thmepsilon} applies. If $G$ is finite, then $C(G,S)=\infty$ for any symmetric generating set (this is a tautology: choosing $\alpha$ larger than $|G|-1$, we see that, for any non-empty $\Omega$, there is no radius $r$ such that $|B(r)|>(1+\alpha)|\Omega|$; hence $\Phi_S[(1+\alpha)|\Omega|]$ is equal to the infimum of the empty set which is $\infty$).

\begin{defi}[F{\o}lner functions]\label{foldef} Let $(G,S)$ be a finitely generated group with a finite symmetric generating set.
The \textit{F{\o}lner function} $\folner_S$ of $G$ relative to $S$ is defined on $\N$ as

$$\folner_S(n)=\min\left\{|\Omega|:\Omega\subset G,\frac{|\partial_S\Omega|}{|\Omega|}\leq\frac{1}{n}\right\},$$
where $\Omega$ is always assumed to be finite and non-empty. Notice that $\folner_S(1)=1$.
\end{defi}

\begin{defi}[The optimal multiplicative inner constant in the volume growth lower bound of the F{\o}lner function]
Let $(G,S)$ be a finitely generated group with a finite symmetric generating set.
We define 
$$F(G,S)=\sup\left\{c\geq 0:\exists \alpha\geq 0, \exists \rho\geq 0\, \mbox{such that}\,\forall n\in\mathbb N,\folner_S(n)\geq\frac{|B_S(cn-\rho)|}{1+\alpha}\right\}.$$
\end{defi}


In Subsection \ref{proof of proposition equal} below, we prove the following equality:
\begin{prop}[Equality of the optimal multiplicative constants]\label{equal} Let $(G,S)$ be a finitely generated group with a finite symmetric generating set. Then $C(G,S)=F(G,S)$.
\end{prop}

The above equality is potentially useful to compute $C(G,S)$ for finitely generated amenable groups of exponential growth, because for such groups,  there is a simple general asymptotic formula for $F(G,S)$. This is exploited in the next theorem.

\begin{thm}[A formula for $C(G,S)$ for groups of exponential growth]\label{quotient} Let $(G,S)$ be a finitely generated group with a finite symmetric generating set. If the growth of $G$ is exponential, in other words if 
\[
\lim_{n\to\infty}\frac{\ln\left(|B_S(n)|\right)}{n}>0,	
\]
then 
$$C(G,S)=\frac{\liminf_{n\to\infty}\frac{\ln(\folner_S(n))}{n}}{\lim_{n\to\infty}\frac{\ln\left(|B_S(n)|\right)}{n}}.$$
\end{thm}

\begin{remark} Recall that for finitely generating group and finite symmetric generating set $(G,S)$, the sequence
\[
	n\mapsto\ln\left(|B_S(n)|\right),
\]
is sub-additive and bounded below. Hence Fekete  Lemma applies (see for example \cite[Proposition 3.2]{Man}) and implies that
\[
\lim_{n\to\infty}\frac{\ln\left(|B_S(n)|\right)}{n}=\inf_n\frac{\ln\left(|B_S(n)|\right)}{n}.	
\]
Notice that the limit is finite and non-negative. It is strictly positive if and only if the group has exponential growth.

Handling  the sequence $$n\mapsto \frac{\ln\folner_S(n)}{n}$$ requires some care. Firstly, the sequence may include infinite values (this is the case if and only if $G$ is non-amenable). Secondly, it may take a given finite value an infinite number of times, and containing oscillations with unbounded amplitudes. This is the case in \cite[Example~3.8(2)]{ErsZhe} for $\alpha=1$ and $\beta=2$, as we shortly explain.
Take a strictly increasing sequence of integer $(\eta_i)$ and a function $\tau(n)=n^\alpha$ for $n\in[\eta_{2j-1},\eta_{2j}]$ and $\tau(n)=n^\beta$ for $n\in[\eta_{2j},\eta_{2j+1}]$.
According to \cite[Example~3.8(2)]{ErsZhe} we obtain a finitely generated group with a finite symmetric generating set $(G,S)$, whose F{\o}lner function verifies, for all $n\in\mathbb N$, 
$$Cn\exp(n+\tau(n))\geq\folner_S(n)\geq\exp\left(\frac{1}{C}\left(n+\tau\left(n/C\right)\right)\right),$$
where $C>1$ is a constant.
Hence, if $n\in[\eta_{2j-1},\eta_{2j}]$, we have $$0\leq\frac{\ln\left(\folner(n)\right)}{n}\leq\frac{\ln(Cn)}{n}+1+\frac{\tau(n)}{n},$$ which is smaller than $3$ if $n$ is large enough.
On the other hand, if $n\in[\eta_{2j},\eta_{2j+1}]$, then $$\frac{\ln\left(\folner(n)\right)}{n}\geq\frac{1}{Cn}(n+\tau(n/C))=\frac{1}{C}+\frac{n}{C^3}.$$
The above comments explain why we have to consider
\[
\liminf_{n\to\infty}\frac{\ln(\folner_S(n))}{n}\in[0,\infty].	
\]
\end{remark}

\subsection{Questions}

In \cite{Stan}, it is proved that the lamplighter group $G$ with the switch-walk-switch generating set $S$, satisfies $C(G,S)\leq 2$.  Combining this fact with Theorem \ref{thmepsilon} we thus obtain:

\begin{thm}[Framing the optimal universal multiplicative outer constant in the Coulhon Saloff-Coste inequality]\label{thmc0}
The infimum $\inf{C(G,S)}$ over all couples $(G,S)$, where $G$ is a finitely generated group and $S$ is a finite symmetric generating set of $G$ satisfies
\[
	1\leq\inf_{(G,S)}{C(G,S)}\leq 2.
\]
\end{thm}
We neither know if there exists a finitely generated group $G$ with a finite symmetric generating set $S$ such that $C(G,S)<2$, nor if it is possible to improve the isoperimetric inequalities presented in this paper, substantially enough, to show that  
$1<\inf_{(G,S)}{C(G,S)}$.




\subsection{Acknowledgements}
We are grateful to Anna Erschler who encouraged us to write down these results.

\section{Proofs}\label{proofs}
\subsection{Preliminary to the proofs: growth of balls and mass transportation.}

When a finite generating symmetric set $S$ is fixed we will sometimes drop the subscript $S$ in the notation $\partial_S\Omega$, $d_S(x,y)$, et caetera.
Let
$$b_r=|B(r)|,s_r=|B(r)\setminus B(r-1)|$$
be the cardinality of the ball, respectively of the sphere, of radius $r$.

\begin{lemma}[Successive spheres]\label{lemmaspheres}
	If $r\geq2$ then $s_r\leq(|S|-1)s_{r-1}$.
\end{lemma}

\begin{proof}For each point $x\in S(r)$ we choose a point $p_r(x)\in S(r-1)$ such that $d(x,p_r(x))=1$.
Each fiber of the map $p_r:S(r)\mapsto S(r-1)$ contains at most $|S|-1$ points.
Indeed, as $r\geq2$, $p_{r-1}p_r (x)=z$ for some $z\in S(r-2)$ and there exists $t\in S$ such that $zt=p_r(x)$.
Hence, if $y\in p_r^{-1}(p_r(x))$, then $y=p_r(x)s$ for some $s\in S\setminus\{t\}$.
\end{proof}

\begin{lemma}[Growth of balls]\label{lemmaballs}
	If $r\geq2$ then $b_r\leq|S|b_{r-1}.$
\end{lemma}

\begin{proof}
Lemma~\ref{lemmaspheres} implies that
\begin{equation*}
\begin{split}
b_r & =b_{r-1}+s_r\leq b_{r-1}+(|S|-1)s_{r-1} \\
	& =b_{r-1}\left(1+(|S|-1)\frac{s_{r-1}}{b_{r-1}}\right) \\
	& \leq b_{r-1}(1+(|S|-1))=b_{r-1}|S|.
	\end{split}
\end{equation*}
\end{proof}

\begin{defi}[{Mass transportation, compare \cite[pages~347-348]{Gro}}]
Let $\Omega\subset G$.
Let $g\in G$.
Let

$$\Omega_g=\{x\in\Omega:xg\in G\setminus\Omega\}.$$

Let
$$g=s_1\cdots s_n$$
with $s_1,\dots,s_n\in S$ and $n=d(e,g)$.
Let $g_0=e$ and for $1\leq k\leq n$, let
$$g_k=s_1\cdots s_k.$$
The first exit point of $x\in\Omega_g$ with respect to the path $(g_k)_{0\leq k\leq n}$ is
$$E_g(x)=g_k$$
where $0\leq k\leq n$ is minimal such that $g_k\in\partial\Omega$.
\end{defi}
Notice that $E_g$ depends not only on $g$ but also on the chosen minimal length expression $s_1\cdots s_n$.
Nevertheless we will simply write $E_g$.

\begin{lemma}[{See \cite[pages~347-348]{Gro}}]\label{lemmatransport}
Let $g\in G$.
Let $\Omega$ be a finite subset of $G$.
We have:
$$|\Omega_g|\leq|g||\partial\Omega|.$$
\end{lemma}
\begin{proof}
Consider the map
$$E_g:\Omega_g\mapsto\partial\Omega.$$
To prove the lemma it is enough to show the following.
For any $b\in\partial\Omega$,
$$|E_g^{-1}(b)|\leq|g|.$$
To prove this inequality, let $x\in\Omega_g$ such that $E_g(x)=b$ (if such an $x$ does not exist the inequality is trivially true).
By definition, $xg_{k(x)}=b$.
If $y\in\Omega_g$ is such that $E_g(y)=b$, then
$$y=xg_{k(x)}g_{k(y)}^{-1}$$
is completely determined by the value of $0\leq k(y)\leq n-1$.
\end{proof}

For each $x\in\Omega$ and each $r>0$, we define
$$R_x^r(\Omega)=\{(x,g):g\in G, xg\in G\setminus\Omega, |g|\leq r\}.$$
(It is convenient to visualize an element of $R_x^r(\Omega)$ as a ray from $x\in\Omega$ to a point of $G$ which is not in $\Omega$ and lays at distance at most $r$ from $x$.)

\begin{lemma}[Two ways of counting]\label{lemmatwocount}
For any $r>0$,
$$\sum_{x\in\Omega}|R_x^r(\Omega)|=\sum_{g\in B(r)}|\Omega_g|.$$
\end{lemma}

\begin{proof}
We claim that the map
\begin{gather*}
\bigsqcup_{g\in B(r)}\Omega_g\rightarrow\bigcup_{x\in\Omega}R_x^r(\Omega) \\
x_g\mapsto(x_g,g)
\end{gather*}
is a bijection.
Notice that $\Omega_g\cap\Omega_h$ may be nonempty for some $g\neq h$ in $G$ - this is why we consider the disjoint union of the sets $(\Omega_g)_{g\in G}$ - whereas $R_x^r(\Omega)\cap R_y^r(\Omega)=\emptyset$ if $x\neq y$.
The above map is well defined because if $x_g\in\Omega_g$ with $g\in B(r)$, then $x_gg\in G\setminus\Omega$, hence $(xg,g)\in R_{x_g}^r(\Omega)$.
The map is obviously one-to-one and any element of $\bigcup_{x\in\Omega}R_x^r(\Omega)$ is obviously of the form $(x_g,g)$ with $x_g\in\Omega_g$ and $g\in B(r)$.
Hence the map is also onto.
\end{proof}

\begin{lemma}\label{lemma35}
Let $\alpha\geq 0$ and $\Omega\subset G$.
Assume that there exists $r$ such that $|B(r)|\geq(\alpha+1)|\Omega|$ (this hypothesis is fulfilled for any $\alpha\geq 0$ and any finite $\Omega$, in the case where $G$ is infinite).
Then
$$\alpha|\Omega|\leq|R_x^r(\Omega)|.$$
\end{lemma}
Remark that for $r=\Phi[(1+\alpha)|\Omega|]$, we have $|B(r)|>(\alpha+1)|\Omega|$.
\begin{proof}
The map
\begin{gather*}
R_x^r(\Omega)\rightarrow xB(r)\cap(G\setminus\Omega) \\
(x,g)\mapsto xg
\end{gather*}
is well defined because if $(x,g)\in R_x^r(\Omega)$ then $xg\in G\setminus\Omega$ and $|g|\leq r$.
It is onto because any element of $xB(r)\cap(G\setminus\Omega)$ can be written $xg$ with $|g|\leq r$.
Hence:
$$|R_x^r(\Omega)|\geq|xB(r)\cap(G\setminus\Omega)|.$$
Applying the hypothesis of the lemma and the above inequality, we obtain
\begin{equation*}
\begin{split}
|\Omega|+\alpha|\Omega| & \leq|xB(r)| \\
& =|xB(r)\cap\Omega|+|xB(r)\cap(G\setminus\Omega)|\\
& \leq|\Omega|+|R_x^r(\Omega)|.
\end{split}
\end{equation*}
This proves the lemma.
\end{proof}

\begin{lemma}\label{lemmaconclude}
Let $\alpha\geq 0$.
Let $\Omega$ be a finite subset of $G$.

Assume that $r=\Phi[(1+\alpha)|\Omega|]$ is finite.
Then for any $\varepsilon>0$,
\[	
|R_x^r(\Omega)|\geq\frac{\alpha}{1+\alpha}|B(r-\varepsilon)|.
\]
\end{lemma}

\begin{proof}
According to Lemma~\ref{lemma35} and the definition of $r$,
$$|R_x^r(\Omega)|\geq\alpha|\Omega|\geq\frac{\alpha}{1+\alpha}|B(r-\varepsilon)|.$$
\end{proof}

\subsection{Proofs of Theorems \ref{thmaveragegrowth} and \ref{thmepsilon}}
We begin with the proof of Theorem~\ref{thmaveragegrowth}.
\begin{proof}
Let $\alpha\geq 0$.
Let $\Omega$ be a non-empty finite subset of $G$. Notice that $$\Phi[(1+\alpha)|\Omega|]>0.$$ We assume that
$\Phi[(1+\alpha)|\Omega|]$
is finite. We define $r=\Phi[(1+\alpha)|\Omega|]$.
We now apply successively the definition of $\mathbb E[|X_r|]|$,  Lemma~\ref{lemmatransport}, Lemma~\ref{lemmatwocount}, and Lemma~\ref{lemmaconclude}:
\begin{align*}
|B(r)|\mathbb E[|X_r|]|\partial\Omega|&=\sum_{g\in B(r)}|g||\partial\Omega|\geq\sum_{g\in B(r)}|\Omega_g|\\
&=\sum_{x\in\Omega}|R_x^r(\Omega)|\\
&\geq\frac{\alpha}{1+\alpha}|\Omega||B(r-1)|.
\end{align*}
This proves the theorem.
\end{proof}

We then prove Theorem~\ref{thmepsilon}.

\begin{proof}
Let $\alpha=\frac{1-\varepsilon}{\varepsilon}$. Equivalently,
$\frac{\alpha}{1+\alpha}=1-\varepsilon$.
We choose $r=\Phi[(1+\alpha)|\Omega|]$, which is finite by hypothesis.
Let
$\alpha'=\frac{|B(r)|}{|\Omega|}-1$. Equivalently,
$|B(r)|=(1+\alpha')|\Omega|$.
Notice that the definition of $\Phi$ implies that $\alpha'>\alpha$.
We apply successively Lemma~\ref{lemmatransport}, Lemma~\ref{lemmatwocount}, Lemma~\ref{lemma35}:
\begin{equation*}
\begin{split}
|B(r)|r|\partial\Omega| & \geq\sum_{g\in B(r)}|g||\partial\Omega|\geq\sum_{g\in B(r)}|\Omega_g|=\sum_{x\in\Omega}|R_x^r(\Omega)| \\
& \geq\sum_{x\in\Omega}\alpha'|\Omega|=\sum_{x\in\Omega}\frac{\alpha'}{1+\alpha'}|B(r)|>\frac{\alpha}{1+\alpha}|\Omega||B(r)|.
\end{split}
\end{equation*}
This proves the theorem.
\end{proof}

\subsection{The Coulhon Saloff-Coste inequality in terms of the F{\o}lner function}
\begin{prop}[The Coulhon Saloff-Coste inequality in terms of the F{\o}lner function]\label{CS equivalent to Folner} Let $(G,S)$ be a finitely generated group with a finite symmetric generating set.
\begin{enumerate}
	\item Let $c>0$ and $\alpha\geq 0$. Assume that for any finite non-empty subset $\Omega$ of $G$,
	\[
	\frac{|\partial_S\Omega|}{|\Omega|}\geq c\frac{1}{\Phi_S[(1+\alpha)|\Omega|]}.	
	\]
	Then, for any  stricly positive $\rho>0$, for all $n\in\mathbb N$, 
	$$\folner_S(n)\geq\frac{|B_S(cn-\rho)|}{1+\alpha}.$$
	\item Assume $S\neq\emptyset$ (the case $(G,S)=(\{e\},\emptyset)$ is trivial). Let $c>0$, $\alpha\geq 0$, and $\rho\geq 0$. Assume that for all $n\in\mathbb N$,
	$$\folner_S(n)\geq\frac{|B_S(cn-\rho)|}{1+\alpha}.$$
	Then, for all finite non-empty subset $\Omega$ of $G$,
	\[
	\frac{|\partial_S\Omega|}{|\Omega|}\geq c\frac{1}{\Phi_S\left[|S|^{\left \lceil{\rho+c}\right \rceil}(1+\alpha)|\Omega|\right]},	
	\]
	where $\left \lceil{\rho+c}\right \rceil$ is the smallest integral upper bound of $\rho+c$. 
\end{enumerate}
\end{prop}

\begin{proof} We begin with the proof of the first implication. Let $n\in\mathbb N$. If $\folner(n)=\infty$, there is nothing to prove. If $\folner(n)<\infty$, we choose $\Omega\subset G$ which realizes $\folner(n)$. More precisely, $\Omega$ is a finite non-empty subset of $G$ such that $\folner(n)=|\Omega|$ and such that
\[
\frac{|\partial\Omega|}{|\Omega|}\leq\frac{1}{n}.
\]
Applying the hypothesis of the first implication we thus obtain:
\[
c\frac{1}{\Phi[(1+\alpha)|\Omega|]}\leq\frac{|\partial\Omega|}{|\Omega|}\leq\frac{1}{n}.	
\]
Hence:
$$\Phi[(1+\alpha)|\Omega|]\geq cn.$$
Eventually, the definition of $\Phi$ and then the equality $\folner(n)=|\Omega|$ show that for any $\rho>0$, 
$$|B(cn-\rho)|\leq (1+\alpha)|\Omega|=(1+\alpha)\folner(n).$$
This ends the proof of the first implication. 

In order to prove the second implication,
we start with the special case when $\partial\Omega=\emptyset$. In other words, $\Omega=G$. As $S$ is non-empty by hypothesis, $|S|^{\left \lceil{\rho+c}\right \rceil}(1+\alpha)\geq 1$ and as the infimum of the empty set is infinite, $\Phi[|S|^{\left \lceil{\rho+c}\right \rceil}(1+\alpha)|\Omega|]=\infty$. Hence
	\[
	\frac{|\partial\Omega|}{|\Omega|}=0\geq c\frac{1}{\infty}=c\frac{1}{\Phi[|S|^{\left \lceil{\rho+c}\right \rceil}(1+\alpha)|\Omega|]}.	
	\] 
We now assume $\partial\Omega\neq\emptyset$. Let $n$ be the unique integer such that
$$\frac{1}{n}\geq\frac{|\partial\Omega|}{|\Omega|}>\frac{1}{n+1}.$$
By definition of the F{\o}lner function, $|\Omega|\geq\folner(n)$. 
We consider two cases. First case: $cn-\rho\geq 1$. Hence $r=cn-\rho+1\geq 2$ so that Lemma \ref{lemmaspheres} repeatedly applies and implies the inequality
\[
|B(cn-\rho)|\geq|S|^{-\left \lceil{\rho+c}\right \rceil}|B(cn+c)|.	
\]
Applying the hypothesis, we conclude that
\[
|S|^{\left \lceil{\rho+c}\right \rceil}(1+\alpha)|\Omega|\geq |S|^{\left \lceil{\rho+c}\right \rceil}(1+\alpha)\folner(n)\geq |S|^{\left \lceil{\rho+c}\right \rceil}|B(cn-\rho)|\geq|B(cn+c)|.	
\]
This implies that
\[
\Phi[|S|^{\left \lceil{\rho+c}\right \rceil}(1+\alpha)|\Omega|]\geq cn+c,
\]
and we conclude that
\[
\frac{|\partial\Omega|}{|\Omega|}>\frac{1}{n+1}\geq  c\frac{1}{\Phi_S\left[|S|^{\left \lceil{\rho+c}\right \rceil}(1+\alpha)|\Omega|\right]}.
\]
Second case: $cn-\rho<1$. Hence,
\[
	\frac{|\partial\Omega|}{|\Omega|}>\frac{1}{n+1}>\frac{c}{\rho+c}.
\]
But $\rho+c\leq\Phi[|S|^{\left \lceil{\rho+c}\right \rceil}(1+\alpha)|\Omega|]$ because
$|B(\rho+c)|\leq|S|^{\left \lceil{\rho+c}\right \rceil}\leq|S|^{\left \lceil{\rho+c}\right \rceil}(1+\alpha)|\Omega|$. This finishes the proof of the second case of the second implication.
\end{proof}

\subsection{Proof of Proposition \ref{equal} and Theorem \ref{quotient}: in order to know $C(G,S)$, compute $F(G,S)$.}\label{proof of proposition equal}
We define the two sets
\[
\mathcal C(G,S)=\left\{c\geq 0:\exists \alpha\geq 0\, \mbox{such that}\,\forall\Omega\subset G,\frac{|\partial_S\Omega|}{|\Omega|}\geq c\frac{1}{\Phi_S[(1+\alpha)|\Omega|]}\right\},	
\]
where $\Omega$ is always assumed finite and non-empty,
and
\[
\mathcal F(G,S)=\left\{c\geq 0:\exists \alpha\geq 0, \exists \rho\geq 0\, \mbox{such that}\,\forall n\in\mathbb N,\folner_S(n)\geq\frac{|B_S(cn-\rho)|}{1+\alpha}\right\}.	
\]

We prove Proposition \ref{equal}.
\begin{proof}
The sets $\mathcal C(G,S)$ and $\mathcal F(G,S)$ both contain $0$.
Assume $c>0$ belongs to 	$\mathcal C(G,S)$. Then, according to the first implication of Proposition 
\ref{CS equivalent to Folner}, it also belongs to $\mathcal F(G,S)$. This proves that 
\[
	\sup\mathcal F(G,S)\geq \sup\mathcal C(G,S).
\]
If $G$ is finite (in particular if $G=\{e\}$) it is easy to check that $C(G,S)$ and $F(G,S)$ are both infinite. Hence we may assume $S\neq\emptyset$ so that the second implication of Proposition 
\ref{CS equivalent to Folner} applies. Assume $c>0$ belongs to 	$\mathcal F(G,S)$. Then, according to the second implication of Proposition 
\ref{CS equivalent to Folner} it also belongs to $\mathcal C(G,S)$. This proves that 
\[
	\sup\mathcal C(G,S)\geq \sup\mathcal F(G,S).
\]
\end{proof}
We prove Theorem \ref{quotient}.
\begin{proof} According to Proposition \ref{equal}, it is enough to prove that
\[
\frac{\liminf_{n\to\infty}\frac{\ln(\folner_S(n))}{n}}{\lim_{n\to\infty}\frac{\ln\left(|B_S(n)|\right)}{n}}=F(G,S).	
\]
Let $c\in \mathcal F(G,S)$. We claim that
\[
	\frac{\liminf_{n\to\infty}\frac{\ln(\folner_S(n))}{n}}{\lim_{n\to\infty}\frac{\ln\left(|B_S(n)|\right)}{n}}\geq c.
\]
Notice that the claim implies the inequality
\[
	\frac{\liminf_{n\to\infty}\frac{\ln(\folner_S(n))}{n}}{\lim_{n\to\infty}\frac{\ln\left(|B_S(n)|\right)}{n}}\geq F(G,S).
\]
We prove the claim. In the case $c=0$, the inequality of the claim is obvious because $\folner_S(n)\geq 1$. Hence we may assume $c>0$. By definition of 
$\mathcal F(G,S)$, there exist $\alpha\geq 0, \rho\geq 0$ such that for all $n\in\mathbb N$,
\[
	\folner_S(n)\geq\frac{|B_S(cn-\rho)|}{1+\alpha}.
\]
Hence,
\[
	\frac{\ln\left(\folner_S(n)\right)}{n}\geq -\frac{\ln(1+\alpha)}{n}+c\frac{\ln\left(|B_S(cn-\rho)|\right)}{cn}.
\]
Applying Lemma \ref{lemmaballs} and taking limits ends the proof of the claim. 
Now we claim that if a constant $c\geq 0$ satisfies
\[
	\frac{\liminf_{n\to\infty}\frac{\ln(\folner_S(n))}{n}}{\lim_{n\to\infty}\frac{\ln\left(|B_S(n)|\right)}{n}}>c,
\]
then $c\in \mathcal F(G,S)$. 
Notice that the claim implies the inequality
\[
	\frac{\liminf_{n\to\infty}\frac{\ln(\folner_S(n))}{n}}{\lim_{n\to\infty}\frac{\ln\left(|B_S(n)|\right)}{n}}\leq F(G,S).
\]
We prove the claim. As $0\in \mathcal F(G,S)$, we may assume 
\[
	\frac{\liminf_{n\to\infty}\frac{\ln(\folner_S(n))}{n}}{\lim_{n\to\infty}\frac{\ln\left(|B_S(n)|\right)}{n}}>c>0.
\]
Let $N_0$ such that for any $n\geq N_0$,
\[
	\frac{\frac{\ln\left(\folner_S(n)\right)}{n}}{\frac{\ln\left(|B_S(cn)|\right)}{cn}}\geq c,
\]
equivalently, 
$$\folner_S(n)\geq|B_S(cn)|.$$
By choosing $\alpha\geq 0$  large enough, we obtain:
\[
	\forall n\in\mathbb N,\, \folner_S(n)\geq\frac{|B_S(cn)|}{1+\alpha}. 
\]
Hence, $c\in \mathcal F(G,S)$.
\end{proof}

\end{document}